\documentclass[12pt]{paper}
 \usepackage{geometry}
\usepackage{hyperref}
\usepackage{marginnote}
                \usepackage{amsthm,amsfonts,amsmath,amscd,amssymb,epsfig,verbatim,
}

\usepackage{marginnote}

\usepackage{graphicx}
\usepackage{epstopdf}
 {\title{Flexible Lagrangians} 
 \author{Yakov Eliashberg$^1$\thanks{Partially supported by NSF grant  DMS-1505910}   \and   
 Sheel Ganatra$^1$\thanks{Supported by an NSF postdoctoral fellowship} \and 
 Oleg Lazarev$^1$\thanks{Partially supported by NSF grant  DMS-1505910 and an NSF grant DGE-114747 } \and \\
 {\small ${ }^1$Department of Mathematics, Stanford University}}

\date{}  
 \setlength{\marginparwidth}{1.2in}
\let\oldmarginpar\marginpar
\renewcommand\marginpar[1]{\-\oldmarginpar[\raggedleft\footnotesize #1]%
{\raggedright\footnotesize #1}}

\parindent=0pt
\parskip=4pt
\hyphenation{ma-ni-fold ma-ni-folds sub-ma-ni-fold sub-ma-ni-folds}
\theoremstyle{plain}
\newtheorem{theorem}{Theorem}[section]
\newtheorem{thm}[theorem]{Theorem}
\newtheorem{corollary}[theorem]{Corollary}
\newtheorem{cor}[theorem]{Corollary}

\newtheorem{prop}[theorem]{Proposition}
\newtheorem{lemma}[theorem]{Lemma}

\newtheorem{problem}[theorem]{Problem} 
\theoremstyle{remark}

\newtheorem{remark}[theorem]{Remark}
\newtheorem*{remark*}{Remark}
 
\newtheorem*{example*}{Example}

\theoremstyle{definition}
\newtheorem{definition}[theorem]{Definition}
\renewcommand{\mod}{\mathrm{mod}} 

\newcommand{\wt}{\widetilde}
\newcommand{\wh}{\widehat}

\newcommand{\p}{\partial}

\newcommand{\om}{\omega}

\newcommand{\eps}{\varepsilon}

\newcommand{\Z}{{\mathbb{Z}}}
\newcommand{\R}{{\mathbb{R}}}
\newcommand{\C}{{\mathbb{C}}}

\newcommand{\Int}{{\rm Int\,}} 

\newcommand{\tb}{{\rm tb}}

\newcommand{\Id}{\mathrm {Id}}

\newcommand{\rank}{\mathrm{rank}}

\def\Op{{\mathcal O}{\it p}\,}
\newcommand{\fW}{{\mathfrak W}}

\numberwithin{figure}{section}

\begin{document}
\maketitle

\begin{abstract}
We introduce and discuss  notions of  regularity and flexibility for Lagrangian manifolds with Legendrian boundary in Weinstein domains.  There is a surprising abundance  of flexible  Lagrangians. In turn, this leads    to new constructions of Legendrians submanifolds and Weinstein manifolds. For instance, many closed $n$-manifolds   of dimension $n>2$   can be realized  as  exact Lagrangian submanifolds of $T^*S^n$ with   possibly exotic Weinstein symplectic structures.  These Weinstein structures on $T^* S^n$, infinitely many of which are distinct, are formed by a single handle attachment to the standard $2n$-ball  along the Legendrian boundaries of  flexible Lagrangians. 
We also formulate a number of open problems.
\end{abstract}
\section{Liouville and Weinstein cobordisms}

The main goal of the paper is a discussion of two new notions of regularity and
flexibility for exact Lagrangian cobordisms with Legendrian boundaries in
Weinstein cobordisms,  see Sections \ref{sec:regularlagn} and
\ref{sec:flexiblelagn}.  In particular, we prove an existence $h$-principle for
flexible  Lagrangian cobordisms (Theorem \ref{thm:main}), explore applications
to Lagrangian and Legendrian embeddings and exotic Weinstein structures, and
formulate  throughout the paper numerous open problems.

 \medskip 
 A {\it Liouville cobordism} between contact manifolds
is a $2n$-dimensional cobordism $(W, \p_- W, \p_+W$) equipped with a
pair $(\om, X)$ of a symplectic form and an expanding (Liouville)
vector field for $\om$, i.e. $L_X\om=\om$, which is outward pointing along $\p_+ W$ and inward
pointing along $\p_- W$, such that
 the contact structure  induced by the Liouville form  $\lambda:=\iota(X)\om$ on $\p_\pm W$ coincides with $\xi_\pm$.
 If in addition we are given a Morse function $\phi:W\to\R$ that is {\em defining} for $W$ and {\em Lyapunov} for $X$, i.e.  it attains its minimum on $\p_-W$, its maximum on $\p_+W$ and has no critical points on $\p W$,  and  satisfies the inequality $d\phi(X)\geq c||X||^2$ for some $c>0$, then the triple  $(\om, X,\phi)$ is called 
a {\em Weinstein cobordism} structure on  $W$
between  contact manifolds $(\p W_-,\xi_-)$ and $(\p W_+,\xi_+)$, see  \cite{Eli90,Wei91,EliGro91,CE12}.
 A cobordism with $\p_-W=\varnothing$ will be referred as a {\em Weinstein domain}.
 
 Stable manifolds of  zeroes of $X$ for a Weinstein cobordism structure $(W,\om, X,\phi)$ are necessarily $\om$-isotropic, see \cite{EliGro91}, and in particular the indices of all critical points of $\phi$ are always $\leq n$. A cobordism is called {\em subcritical} if there are are no critical points of index $n$. 

Every Liouville or Weinstein cobordism can be canonically completed by adding {\em cylindrical ends}
  $(-\infty,0]\times\p_-W$ and $[0,\infty)\times \p_+W$ with $X=\frac{\p}{\p s}$ and $\phi$ equal   to $s$ up to an additive constant. Here we denote by $s$ the coordinate corresponding to the first factor.   An important feature of completed Liouville cobordisms, see \cite{EliGro91,CE12}, is that if $(\om_t,X_t), t\in [0,1]$, is a homotopy of completed Liouville cobordisms then there exists an isotopy $\phi_t:W\to W$ such that $\phi_t^*\omega_t=\om_0$, $t\in[0,1]$, which preserves the Liouville field at infinity.
  In other words, the symplectic structure of a completed  Liouville cobordism remains unchanged up to isotopy when one deforms the Liouville structure.
  
  Usually we will not distinguish in the notation between a Weinstein cobordism and its completion.
  Moreover,  
we will allow contact manifolds to have boundaries and will require in this case the cobordisms to be  trivial over boundaries of the contact manifolds. Alternatively, 
a Weinstein cobordism between manifolds $\p_\pm W$ with boundary  can be  viewed as  a sutured manifold with   corner along the suture, see Fig. \ref{fig:satured} (taken from \cite{EM-symp}).
More precisely, we assume that  the boundary  $\p W$ is presented as a union of two manifolds $\p W_-$ and $\p_+W$ with common boundary  $\p^2W=\p_+W\cap \p_-W$, along which it has a corner.  Of course, in this case the function $\phi$ cannot be chosen constant on $\p_-W$ and $\p_+W$.

\begin{figure}[h]
\begin{center}
\includegraphics[scale=.6]{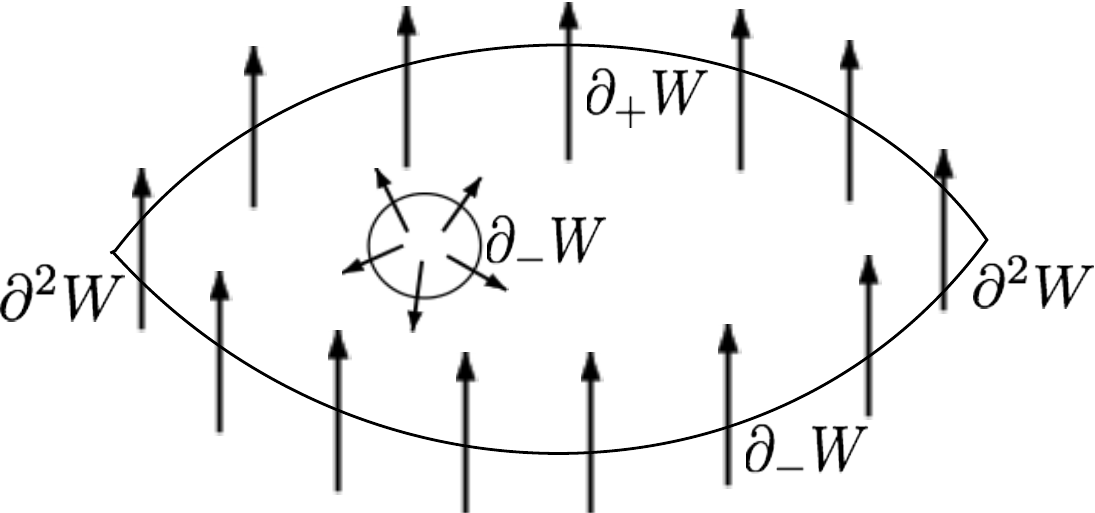}
\caption{A Liouville cobordism $W$ with corners.}\label{fig:satured}
\end{center}  
\end{figure}

Given a $2n$-dimensional Weinstein cobordism  $\fW:=(W,\p_-W,\p_+W;\om,X,\phi)$  we consider in it {\em exact Lagrangian cobordisms with Legendrian boundary}  $ (L,\p_-L,\p_+L)\subset   (W,\p_-W,\p_+W)$. We will additionally require    $L$ to be tangent to $X$ near $\p_\pm L$.  This condition can always  be achieved by a $C^0$-small isotopy of $L$ fixed on $\p_\pm L$. 
 The boundary components $\p_\pm L$ will always   assumed to be closed and contained in $\Int\p_\pm W$. Sometimes we will talk about a {\em parameterized Lagrangian cobordism}, i.e. a diffeomorphism of a smooth cobordism $(L,\p_-L,\p_+L)$ onto an exact  Lagrangian cobordism in $W$ between  Legendrian manifolds  in $\p_\pm W$.
 
 A  Lagrangian cobordism $L$ can be canonically completed to a submanifold  with cylindrical ends in the completion of $W$.
An isotopy between two  exact Lagrangian cobordisms with Legendrian boundaries will be always understood in this class, i.e. as a Hamiltonian isotopy of the completions which at infinity is required to  preserve the Liouville vector field $X$. We note that any  exact  Lagrangian isotopy   with Legendrian boundary  lifts  to a Hamiltonian isotopy of completions.

A Morse decomposition for the Lyapunov function $\phi$ yields an equivalent definition of a Weinstein cobordism as a {\em Weinstein handlebody},  formed by  attaching  handles with symplectically isotropic core discs along contactly isotropic sphere in the regular contact level sets of $\phi$.
A Weinstein cobordism of dimension $2n\geq 6$ is called {\em flexible} if it can
be presented as a Weinstein handlebody so that all critical (i.e. of index $n$) handles are attached along
 {\em loose} Legendrian links, see \cite{CE12, M11} for precise definitions
and discussion. A {\em flexible Weinstein structure} is a choice of such a
presentation.
\section{Regular Lagrangians}\label{sec:regularlagn}
Let $(W, \om, X, \phi)$ be a Weinstein cobordism, and $$ (L,\p_-L,\p_+L) \subset (W, \p_-W,\p_+W,\om, X,
\phi)$$ a Lagrangian cobordism with Legendrian boundary.
\begin{definition}\label{def:regular}
    We say $ L\subset W$ is {\em regular}  if $(W, \om, X, \phi)$ can be deformed  to a Weinstein structure $(W, \om',
    X', \phi')$  through Weinstein structures
    for which $L$ remains Lagrangian, such that the new Liouville vector field
    $X'$ is tangent to $L$. This is equivalent to  the condition $\alpha'|_L=0$, where   $\alpha':=\iota(X')\omega'$ is the corresponding Liouville form.  
\end{definition}


We call such   $(W, \om', X', \phi')$ a {\em tangent Weinstein structure} to
the regular Lagrangian $L$. It follows that all critical points of
$\phi'|_{L}$ are global critical points of $\phi'$, and the local models near such $p$ can be described
by a ``coupled handle attachment'' picture.  
Indeed, let $k$ be the index of  a  critical point  $x$   of   $\phi'$ and $l$ the index of $x$  as a critical point of $\phi'|_L$. We have $l\leq k\leq n$.
 A Weinstein handle of  index $k\leq n$   is isomorphic to the subset $H_k:=\{|p|,|q|,|P|,|Q|\leq 1\}\subset T^*\R^k\times T^* \R^{n-k}$, where  we denoted by $(p,q)$ and $(P,Q)$ the canonical coordinates in $T^*\R^k$ and $T^*\R^{n-k}$.   The handle $H_k$ contains a Lagrangian sub-handle $L_l$ of index $l$ which is the intersection with $H_k$  of   the total space  of the conormal bundle to $$\R^l\subset\R^k\subset T^*\R^k=(T^*\R^k)\times 0\subset T^*\R^k  \subset  T^*\R^k\times T^*\R^{n-k}.$$   When passing  through the critical level  $a=\phi'(x)$ of the critical point $x$  the Weinstein handle  $H_k$ is attached to $\{\phi'\leq a-\eps\}\subset W$ along $\p_-H_k:=H_k\cap  \{|q=1|\}$, and the Lagrangian handle $L_l$ is attached to $\{\phi|_L\leq a-\eps\}\subset L$ along $\p_-L:=  
  \p_-W\cap L_l$.
 
It turns out that given a regular $L$, any Weinstein cobordism structure tangent to $L$ can be further adjusted.
Let 
us call a tangent to $L$ 
Weinstein cobordism structure $(W,\om, X,\phi)$ {\em special}
if there exists a regular value  $c\in\R$  of the function $\phi$ such that
\begin{itemize}\item  all critical points of $\phi$ in the sublevel set $\{\phi \leq c\}$ lie on $L$ and the indices of these critical points for  $\phi$ and $\phi|_L$ coincide;
 \item there are  no critical points of $\phi$  on $L\cap \{\phi\geq c\}$.
\end{itemize}
 In other words,  $(W,\om, X,\phi)$ has the following handlebody presentation. First, one attaches   handles  corresponding to critical points of $\phi|_L$ and then the remaining handles, so that their attaching spheres do not intersect $L$.

For a special tangent to $L$ Weinstein cobordism $(W, \om, X, \phi)$  we  set  $W_L:=\{\phi\geq c\}$  and view $W_L$ as a Weinstein subcobordism  of $(W, \om, X, \phi)$ with the induced  Weinstein structure. We call $W_L$ the {\em complementary Weinstein cobordism to $L$} and note that $\phi$ determines a presentation $W:= T^* L \cup W_L$, where $T^* L$ is endowed with its canonical Weinstein structure. The following lemma asserts that up to homotopy of Weinstein structures for which $L$ remains Lagrangian, the existence of such a presentation is equivalent to regularity:
 
\begin{lemma} \label{lm:equiv-charact-reg}
Let $(L,\p_-L,\p_+L)\subset (W,\p_-W,\p_+W)$
be an exact  Lagrangian  subcobordism in a Weinstein  cobordism  $(W,\om_0,X_0,\phi_0)$  tangent to $L$. Then there is a homotopy $(W,\om_t,X_t,\phi_t)$ of tangent to $L$ Weinstein structures     such that  $(W,\om_1,X_1,\phi_1)$ is special.
     
\end{lemma}
\begin{proof}
Suppose first that  for each  critical point $p$ of the function $\phi|_L$ its  index on $L$ coincides with its index as a  critical point of $\phi$ on the whole $W$.\footnote{In this case the modification of the Lagrangian after passing through the corresponding critical value coincides with the {\em ambient Legendrian surgery} defined by G. Dimitroglou Rizell in \cite{Ri12}.} 
Then  for any critical point $p\in L$ of $\phi$ its stable manifold is contained in $L$, and hence for any critical point $q\notin L$ there are no $X$-trajectories converging to $q$ at the negative direction, and to $p$ at  the positive one.
Hence, using Lemma 9.45 from \cite{CE12}  we can deform $\phi$ without changing $X$ and $\om$ (and hence keeping Weinstein structure tangent to $L$)
so that the critical values corresponding to critical points from $L$ are all smaller than  the critical values corresponding to critical points of $\phi$ which are not in $L$. Then an intermediate regular value $c$ has the required properties, and thus the Weinstein structure is special.

Suppose now that the index $l$ of a  critical point $p$ of the function
$\phi|_L$ is less than the index   $k$ of $p$  for $\phi$ on the whole $W$.
Let $D^k$ be the stable disc of $p$ on $W$ and $D^l=D^k\cap L$ the stable disc of
$p$ for the function $\phi|_L$.
Note that  that there exists a function $\wt \phi :D^k\to \R$ which coincides with $\phi$ on $ D^l\cup \p D^k $, has a critical point  of index $l$ at $p$ and two additional critical points $p', p''$ on  $D^k\setminus D^l$ of indices $l+1$ and $k$, respectively. Hence the attaching of one handle of index $k$ corresponding  to the point $p$ can be replaced by attaching of three handles of indices $l,l+1$, and $k$ corresponding to the points $p, p'$, and $p''$. Moreover, only the  first handle intersects the Lagrangian $L$. Hence,  the claim follows from the already considered case.\end{proof}

The following proposition characterizes regular Lagrangian discs.
\begin{prop}\label{prop:reg-discs} Let $(C,\p C)\subset (W,\p_+W)$ be a  Lagrangian disc with Legendrian boundary. It is regular if and only if there is a Weinstein handlebody representation of the  cobordism $W$ for which $L$ coincides with the co-core Lagrangian disc of one of the index $n$ handles.
\end{prop}
\begin{proof} If $C$ is a co-core disc for a Weinstein structure, then this structure is tangent to it, and hence $C$ is regular. Conversely, if $C$ is regular for a Weinstein structure $\fW$, then by Lemma \ref{lm:equiv-charact-reg}  this structure (after Weinstein homotopy of tangent to $C$ Weinstein structures) admits a Weinstein handlebody consisting of the ball $B^{2n}$ with Lagrangian equatorial disc  $C\subset B$ and other Weinsteins handles glued to $\p B\setminus \p C)$. One can deform the Weinstein structure on $B^{2n}$ by creating two additional critical points of index $n$ and $(n-1)$ such that $C$ serves as the co-core disc of the corresponding index $n$ handle. It remains to note that using 
Proposition 10.10  from \cite{CE12}
one can re-order critical points of a Lyapunov function  so that the handle corresponding to the critical point $C$ is the last one to attach.

\end{proof}

The regularity property for Lagrangian submanifolds also has (at least conjecturally) a Lefschetz fibration characterization. E.~Giroux and J.~Pardon have suggested to us that the following can probably be proven along the lines of \cite{GP15}, adapting results of \cite{AMP}: {\em  for any regular Lagrangian submanifold $L\subset W$ with Legendrian boundary there exists a Lefschetz fibration  over $\C$ which projects $L$ to a ray in $ \R\subset \C$.} Of course, the converse statement is true: a Lagrangian with such a Lefschetz presentation is regular.

There are many other natural examples of regular Lagrangians, including the zero
section and cotangent fibers of a cotangent bundle, and more generally smooth
loci of Lagrangian skeleta or ascending Lagrangian co-cores of the flow of $X$.
A necessary condition for the  regularity of a closed  $L$, or more generally of a cobordism $L$ with   $\p_+L=\varnothing$,   is given by the following:
\begin{lemma}\label{lm:nec-reg}
Let $L$ be a  regular Lagrangian cobordism with $\p_+L=\varnothing$. Then 
  the inclusion $H_n (L,\p_-L )\to H_n(W,\p_-W)$ is injective. Here the homology is taken with integer coefficients if $L$ is orientable and with $\Z/2$-coefficients otherwise.   Moreover, in the orientable case the image of the  (relative) fundamental class of  $(L,\p_-L)$ in   $H_n(W,\p_-W)$ is indivisible.
\end{lemma}
\begin{proof}
Assuming the Weinstein structure is special, we observe that a generic fiber $F$ of the cotangent bundle  $T^*L$ has boundary $\p F$ which does not intersect the attaching spheres of any of the additional Weinstein handles, and hence $F$ represents a homology class $[F]\in H_n(W,\p_+W)$. But $[F]\cdot[L]=\pm 1$, and hence the class $[L]\in H_n(W,\p_-W)$ is indivisible.
\end{proof}
The results of \cite{EM-caps} and  \cite{M13} show that this injectivity condition {\em
does not necessarily } hold when 
$ \p_-L $ is loose, or when $ \p_-L =\varnothing $ 
but $(\p_-W,\xi_-)$ is overtwisted, and these therefore provide examples of non-regular Lagrangian cobordisms. 

However, we do not know any counterexample to the positive answer to the following problem:
 
\begin{problem}\label{prob:Liou-Wein}
   Suppose $\p_- W = \varnothing$ (or more generally when  $(\p_-W,\xi_-)$ is
    tight and $\p_-L$ is not loose). Is every Lagrangian cobordism $L\subset W$ regular? In
    particular, does  the conclusion of  Lemma \ref{lm:nec-reg}   hold for such
    $L$? For instance, is the image of the fundamental class of a closed
    Lagrangian manifold in a Weinstein manifold necessarily indivisible (and in
    particular non-zero)?  
\end{problem}

\section{Flexible Lagrangians}\label{sec:flexiblelagn}

\begin{definition}\label{def:flexible}
    We say a Lagrangian cobordism $L\subset W$ is {\em flexible} if it is regular with a
   special tangent Weinstein structure $(W, \om, X, \phi)$  
   for which the complementary Weinstein cobordism $(W_L,\om|_{W_L}, X|_{W_L},\phi|_{W_L})$ is flexible.
\end{definition} 
In the case when $\p_+L\neq\varnothing$ one can equivalently characterize flexibility in terms of tangent but not necessarily special Weinstein structures.
 \begin{lemma}\label{lm:equiv-charact-flexibility}
 A Lagrangian cobordism $L\subset W$  with   $\p_+ L\neq\varnothing$ is flexible if and only it is regular with a tangent to $L$ Weinstein cobordism  structure which admits a partition into elementary cobordisms such that links  of attaching spheres of  index $n$ handles are loose in the complement of $L$.
 \end{lemma}
 \begin{proof}
The proof repeats the steps of the proof of Lemma \ref{lm:equiv-charact-reg}.
If  for each  critical point $p$ of the function $\phi|_L$ its  index on $L$ coincides with its index as a  critical point of $\phi$ on $W$, then we modify the Weinstein structure into one which is special and tangent to $L$ without changing $\om$ and $X$. For the resulting Weinstein structure the cobordism $W_L$ is automatically flexible.
 
Suppose that the index $l$ of a  critical point point $p$ of the function
$\phi|_L$ is strictly less than the index   $k$ of $p$  for $\phi$ on $W$.
Letting $D^k$ be the stable disc of $p$ on $W$ and $D^l=D\cap L$ the stable disc of
$p$ for the function $\phi|_L$, we modify, as in the proof of Lemma
\ref{lm:equiv-charact-reg} the Weinstein structure by changing the index of $p$ for $\phi$ to $l$ at the expense of creating two new critical points of index $l+1$ and $k$ respectively on the stable  $D^k$.

Finally we observe  that if $k=n$, then   the index $n$ handle corresponding to the point $p''$ is attached along a loose Legendrian by assumption. If $l=n-1$,   the second  index $n$ handle corresponding to the point $p'$ is in canceling position with the $n-1$ index handle corresponding to $p$ 
and hence is also attached along  a loose knot. This implies the flexibility of $L$.

 The opposite implication is straightforward.
 \end{proof}

The next proposition gives two fundamental examples of flexible  Lagrangian submanifolds. Recall that a product  $\fW_1\times \fW_2=(W_1\times W_2, \om_1\oplus\om_2, X_1\oplus X_2,\phi_1\oplus\phi_2)$ of  completed Weinstein cobordisms $\fW_1=(W_1,\om_1,X_1,\phi_1)$ and  $\fW_2=(W_2,\om_2,X_2,\phi_2)$  is again a (completed) Weinstein cobordism  (of manifolds with boundary). 
For a Weinstein cobordism  structure $\fW=(W,\om, X,\phi)$ we denote
$\overline\fW:=(W,-\om, X,\phi)$ and observe that the the structure
$\fW\times\overline\fW$ is tangent to the Lagrangian diagonal $D\subset W\times W$, and hence  $D$ is regular for $\fW\times\overline\fW$.
\begin{prop}\label{prop:flex}
\begin{enumerate}
\item Let $(W,\p_- W,\p_+W)$ be a  flexible Weinstein cobordism and let $\wt W$ denote the result of attaching an $n$-handle to $W$ along a loose  Legendrian knot $\Lambda\subset \p_+W$. Then the co-core disc  $C$ of the attached handle is flexible.
\item   Let $\fW=(W,\om, X,\phi)$  be  a {\em flexible} Weinstein cobordism structure.  Then the diagonal $D$ is flexible for  $\fW\times\overline\fW$.
 \end{enumerate}
\end{prop}
  \begin{proof}
      (i)  This is immediate from Lemma \ref{lm:equiv-charact-flexibility}.

(ii) To show that $D$ is flexible, we must
show that all index $2n$ handles determined by $\phi \oplus \phi$ are attached
along Legendrians which are loose in the complement of $D$. 
The  proof of this fact is essentially identical   to the proof  that the product of a flexible Weinstein manifold with any other
Weinstein manifold is always flexible (and moreoever, in the case of $\fW\times\overline\fW$, 
one can ensure by construction that the loose charts stay away from $D$). This folkloric  statement  was known for a while  to several specialists and recently was proven by Murphy-Siegel \cite[Proposition 3.7]{MS15}.
%
%
 \end{proof}
\begin{problem}\label{prob:converse-to-flex}
Is the converse to each  of the statements in Proposition \ref{prop:flex} true? In other words, is it true that
\begin{enumerate}
\item If a Lagrangian co-core of an $n$-handle  is flexible, then its attaching Legendrian sphere is loose? Moreover, can a flexible Weinstein cobordism remain
flexible after attaching an $n$-handle along a non-loose  Legendrian knot? 
\item If the diagonal in the product $\fW\times\overline\fW$
    is flexible, then 
    $\fW$ is flexible.
\end{enumerate}
\end{problem}

\begin{problem}\label{prob:around-nearby}
  ({\em around the nearby Lagrangian conjecture})
  \begin{enumerate}
 \item  Are all regular closed Lagrangians in $T^*M$ 
  Hamiltonian isotopic? 

 \item  Let $L\subset W$ be a flexible closed Lagrangian in a Weinstein domain $W$. Are all (regular) closed Lagrangian submanifolds  in $W$ Hamiltonian isotopic to $L$?
   {\rm  To tie this to (i), we note that the $0$-section in a cotangent bundle $T^* M$ is tautologically flexible. } 
   \end{enumerate}
\end{problem}

\section{Existence and classification of flexible Lagrangians}\label{sec:proof}
By a {\em  formal parameterized Lagrangian cobordism} in $(W,\p_-W,\p_+W)$ we mean a 
pair
 $(f,\Phi_t)$, where $f:(L,\p_-L,\p_+L)\to (W,\p_+W, \p_-W)$ is  a smooth embedding of an $n$-dimensional  cobordism $(L,\p_-L,\p_+L)$, and   $\Phi_t:TL\to TW$, $t\in[0,1]$,  is a homotopy of injective homomorphisms   such that
\begin{description} 
\item{(i)} $\Phi_0=df$;
\item {(ii)} $\Phi_1$ is a Lagrangian homomorphism, i.e. $\Phi_1(T_xL)\subset T_xW$ is a  Lagrangian  subspace for all $x\in L$;
\item{(iii)}  $\Phi_1|_{TL|_{\p_\pm L}}\subset   {\rm Span}(X,\xi_\pm)$;
\item{(iv)} $\Phi_t(T (\p L)) \subset T (\p W)$; and
\item{(v)} $\Phi_t|_{TL|_{\p L}}$ is transverse to $\p W$ for all $t\in[0,1]$.
\end{description}
A genuine  parameterized Lagrangian  cobordism
$f:(L,\p_-L,\p_+L)\to (W,\p_+W, \p_-W)$ can be viewed as formal by setting $\Phi_t\equiv df$, $t\in[0,1]$.
Two formal  parameterized Lagrangian cobordisms are formally Lagrangian isotopic if they are isotopic through formal   parameterized Lagrangian cobordisms. 
Note that a formal parameterized Lagrangian cobordism has well defined {\em formal 
Legendrian classes} of  its positive and negative boundaries.
If $(L,\p_-L,\p_+L)\subset (W,\p_-W,\p_+W)$ is a subcobordism and $f$ is the inclusion map $(L,\p_-L,\p_+L)\hookrightarrow (W,\p_-W,\p_+W)$  then we will  drop the word ``parameterized" from the term  ``formal parameterized Lagrangian cobordism".
\begin{remark}\label{rem:2-formal}
One can define a weaker notion of a formal  parameterized Lagrangian cobordism by dropping the transversality condition (v) from the definition. We note, however,  that every homotopy class of weak formal Lagrangians contains a unique homotopy class of strong ones. Indeed, the obstructions to deform a weak to a strong one lie  in the homotopy groups $\pi_k(S^{2n-1})$, where $k\leq n$. But $n<2n-1$ for $n>1$, and hence obstructions vanish.
\end{remark}
\begin{thm}\label{thm:main}
\begin{enumerate}
     \item  { \em (Existence)}
In a flexible  Weinstein cobordism $(W, \om)$, any  formal Lagrangian cobordism  with  non-empty positive boundary of each of its components
is formally Lagrangian isotopic to a   flexible genuine  Lagrangian cobordism.

\item   {\em  (Uniqueness)} Let
\begin{align*}
&f_0: (L_0,\p_-L_0,\p_+,L_0)\subset (W_0,\p_-W_0,\p_+W_0)\;\;\hbox{and}\;\\
 &f_1:(L_1,\p_-L_1,\p_+,L_1)\subset (W_1,\p_-W_1,\p_+W_1)
\end{align*}
 be  two flexible Lagrangian cobordisms.
Then given  diffeomorphisms $h:W_0\to W_1$  and $g:L_0\to L_1$ such that
$h\circ f_0=f_1\circ g$ and   $f^*\om_1$ is homotopic to $\om_0$ via a homotopy of non-degenerate not necessarily closed $2$-forms vanishing on $f_0(L_0)$, there exists an isotopy of $h$   to a symplectomorphism $\wt h :W_0\to W_1$ such that  $\wt h\circ f_0=f_1\circ g$ . 
\end{enumerate}
    \end{thm}

\begin{remark}\begin{enumerate}
\item Note that the condition  $\p_+ L\neq0$ in (i)  is essential. For example, a flexible Weinstein manifold has no compact exact Lagrangians, even though it may have many formal compact Lagrangians. 

\item
    If  the formal embedding in Theorem \ref{thm:main}(i) is genuine Lagrangian near $\p_-L$  then  the formal homotopy can be constructed  fixed near $\p_-L$.

\item
We do not know whether the uniqueness part of Theorem \ref{thm:main} holds up to a  {\em Hamiltonian isotopy}, i.e. whether formally Lagrangian isotopic flexible Lagrangians are Hamiltonian isotopic.\footnote{Proposition \ref{prop:semi-flex-discs} below provides a partial result in this direction.}  
\item
The formal Lagrangian  isotopy  in Theorem \ref{thm:main}(i) need not be $C^0$-small, because the proof relies on the 
 flexibility of the ambient structure. The situation here is similar
  to that of Theorem 7.19 of \cite{CE12}: any formal Legendrian embedding in an overtwisted contact manifold is formally Legendrian isotopic to a genuine Legendrian embedding, but the isotopy need not be $C^0$-small.

  \end{enumerate}
  \end{remark}
 
  \begin{proof}[Proof of Theorem \ref{thm:main}]
      
To prove the existence part,
 let $(j,\Phi_t):(L,\p_-L,\p_+L)\hookrightarrow(W,\p_-W,\p_+W)$ be a formal Lagrangian submanifold and $(U,\eta)$ denote  a tubular  neighborhood of  $L$ in $T^*L$ with its canonical symplectic structure $\eta=d(pdq)$.
 There exist  an extension of $j$ to an embedding $\wh j :U\to W$  and   a homotopy  $\wh \Phi_t:TU\to TW$, $t\in[0,1]$, of fiberwise isomorphisms extending 
 $\Phi_t$, 
 such that $\wh \Phi_0=d\wh j$ and  $\wh \Phi_1$ is a symplectic bundle isomorphism $(TU,\eta)\to (TW,\om)$.  The homotopy $(\Phi_t)_*\eta$ (:= $(\Phi_t^{-1})^* \eta$) of   non-degenerate but not necessarily closed $2$-forms on $\wh j(U)$ extends to  a homotopy $\om_t$, $t\in[0,1]$, of non-degenerate   $2$-forms on $W$ such that $\om_1=\om$ and $\wh j^*\om_0=\eta$.    In particuar, $\om_0$ is a genuine symplectic structure on a neighborhood of $L$ and $ L$ is Lagrangian. We denote by $W_L:= (W_L, (\omega_0)|_{W_L})$ the complement of $\wh j(U)$ with its induced formal symplectic structure. We wish to now show that $W_L$ admits a Weinstein structure in the same almost symplectic class as $(\omega_0)|_{W_L}$ relative boundary.

The condition $\p_+ L\neq \varnothing$ for each component of $L$ implies that $\pi_j(W_{  L},\p_+W_{  L})=0$ for $j\leq n-1$.
Indeed, the pair $(W,\p_+W)$ is $(n-1)$-connected because the cobordism $W$ is Weinstein. 
Hence,    any relative spheroid $\psi: (D^j,\p D^j)\to (W_{  L},\p_+W_{  L})$ extends for $j\leq n-1$ to a spheroid $\Psi:(D^{j+1}_+,\p_-D_+^{j+1})\to (W,\p_+W)$, 
where we denoted  $$D^{j+1}_+:=\left\{\sum\limits_{k=1}^{j+1} x_k^2 \leq 1,\, x_{j+1}\geq0\right\}\subset\R^{j+1},\; \p_-D_+^{j+1}=D^{j+1}_+\cap\{x_{j+1}=0\}$$ and identified the  the upper-half sphere 
$\left\{\sum\limits_{k=1}^{j+1}x_k^2 = 1,\, x_{j+1}=0\right\}=\p D^{j+1}_+\setminus \Int(\p_-D^{j+1}_+)$ with the disc $ D^j$. 
 Assuming without loss of generality that $\Psi$ is transverse to $L$ we conclude that if $j<n-1$ then $\Psi(D^{j+1}_+)\subset W_L$, and hence the homotopy class  $[\psi]\in \pi_j(W_L,\p_+W_{ L})$ is trivial, and if $j=n-1$ the image 
 $\Psi(D^{j+1}_+)$ intersects $L$ transversely in finitely many points.
 Hence, $\pi_{n-1}(W_{ L},\p_+W_{  L})$ is generated by   small $(n-1)$-spheres $S$ linked with $L$ and transported to the base point in $\p_+W_{L}$ by some paths in $W_{  L}$.
 Moreover, the condition $\p_+L\neq\varnothing$ for each connected component of  $L$ allows us to choose $S\subset\p_+W_{ L}$ and the condition $\pi_1(W_{L},\p_+W_L))=0$ provides a homotopy of the connecting path to $\p_+W_{ }$, i.e. the generating relative  spheroid is trivial in $\pi_{n-1}(W_{ L},\p_+W_{ L})$. Hence, the pair $(W_{ L},\p_+W_{  L})$ is $(n-1)$-connected. Then the classical  Whitehead-Smale's handle exchange argument, see \cite{Smale}, allows us to construct a defining function on the cobordism $W_{  L}$ without critical points of index $\geq n$.
 
  Using  Theorem 13.1 from  \cite{CE12}, which holds for cobordisms with corners, 
  we construct a flexible Weinstein cobordism structure on $W_{L}$ which  agrees with the standard symplectic structure on (the boundary of) the given neighborhood $\Op L:= \wh j(U)$ of the Lagrangian $ L$. Together with the canonical subcritical Weinstein structure on $\Op L$, it yields a flexible Weinstein structure $\eta$ on $W$ which is in the same almost symplectic homotopy class as the original symplectic structure on $W$. Using Theorem 14.3 and Proposition 11.8  from \cite{CE12}, we can   
 construct  a diffeotopy $h_t:W\to W$ connecting  the identity  $h_0=\Id$ with a symplectomorphism $h_1:(W,\eta)\to(W,\om)$. Then the parameterized Lagrangian cobordism  $h_1\circ f:(L,\p_-L,\p_+L)\to(W,\p_-W,\p_+W)$ is  in the prescribed formal class.
     
 To prove the second part denote  $\wh L_0:=f_0(L_0)$ and $\wh L_1:=f_1(L_1)$.
Let us observe that
 the Lagrangian neighborhood theorem allows us to    assume that $h$ is symplectic on $\Op \wh L_0$. Let $\fW_0$, and $\fW_1$ denote the given  Weinstein structures on $W_0$ and $W_1$. By assumption the Weinstein structures $h_*\fW_0$   and $\fW_1$ restricted to $W_{\wh L_1}$  are in the 
 same   relative  to  $\p_-W$   almost symplectic class, and hence according to  Theorem   14.3 and Proposition 11.8 from  \cite{CE12}  $h$ is isotopic to a symplectomorphism  $\wt h: W\to W$ via an isotopy fixed on $\Op L_0$, and in particular, we have $\wt h\circ f_0=g\circ f_1$.
\end{proof}
%
\begin{remark}\label{remark:boundaryofflexible}
An interesting aspect of Theorem \ref{thm:main} is that when $\p_- W
=\varnothing$, the positive Legendrian boundaries of  a flexible Lagrangians  $L$ necessarily cannot
be {\it loose} in the sense of \cite{M11} (as they are filled by exact
Lagrangians), and indeed must have non-trivial holomorphic curve invariants. For
instance, the wrapped Floer homology $WFH_*(L,L; W)$ of $L$ must be 0 by i.e., Lemma
\ref{vanishingwrapped}, or more directly, one can note that $W$ is flexible, hence
$SH_*(W)=0$ and $WFH_*(L,L; W)$ is a unital module over $SH_*(M)$.  It follows
\cite{E12} that there is an isomorphism between the (linearized) Legendrian
contact homology $WFH^+_*(\partial L)$ and the relative homology $H_*(L,
\partial L)$. See Problem \ref{prob:memory} for further discussion.
\end{remark}

\begin{corollary}\label{cor-3-manifold} 
Let $L$ be an  $n$-manifold with
non-empty boundary, equipped with a fixed trivialization $\eta$ of its
complexified tangent bundle $TL\otimes\C$. 
Then there exists a flexible Lagrangian embedding with Legendrian boundary
$(L,\p L)\to (B^{2n},\p B^{2n})$ where $B^{2n}$ is the standard symplectic
$2n$-ball,  realizing the trivialization $\eta$. In particular, any 
  $3$-manifold with boundary can be realized as a flexible Lagrangian submanifold of $B^6$ with Legendrian boundary in $\p B^6$.
\end{corollary}
\begin{proof}
    With respect to a reference trivialization $f: TL \otimes \C \cong L \times
    \C^n$, the trivialization $\eta$ is equivalent to the data of a   map
    $\phi:L\to U(n)$ such that $\phi(\p L)\subset U(n-1)\subset U(n)$. Given
    such data, Gromov's h-principle for Lagrangian immersions produces a
    Lagrangian immersion $f:(L,\p L)\to (B^{2n},\p B^{2n})$ transverse to $\p
    B^{2n}$.
  Moreover, using Whitney's cancellation technique and the fact that $\p
  L\neq\varnothing$ we can regularly (but not symplectically) homotope $f$ to
  an embedding $f'$.  The resulting embedding $f'$ inherits a formal
  Lagrangian structure from the immersion $f$.  We then complete the proof
  using Theorem \ref{thm:main}.

\end{proof}
To explore further consequences  of the above constructions we first recall a theorem of Mich\`ele Audin. Given a connected closed n-dimensional manifold  $L$ and an immersion $f:L\to\R^{2n}$ with transverse double points we denote by $d(f)$ the algebraic count of double points. This is an integer if $L$ is orientable and $n$ is even  and  an element of $\Z/2$ otherwise. $d(f)$ is an invariant of the regular homotopy class of $f$ which vanishes if and only if the class contains an embedding.   For a  closed connected  manifold $L$  of dimension $n=2k+1$ we denote by $\chi_{\frac12}(L)$ Kervaire's semi-characteristic 
\[
\chi_{\frac 12}(L):=\sum\limits_{i=0}^k \rank H_i(L)\ (\mathrm{mod}\ 2).
\]
A relationship between $\chi_{\frac 12}(L)$ and $d(L)$ is given (in nice cases) by the following result:
\begin{theorem}\label{thm:Audin}{\rm[M. Audin, \cite{Audin}]}
Let $L$ be a closed  manifold of odd dimension $\neq 1,3$.
Then for any Lagrangian immersion $f:L\to\R^{2n}$ with transverse double points, we have  $d(f)=\chi_{\frac12}(L)$ at least in the following cases:
\begin{enumerate}
\item $L$ is stably parallelizable;
\item $n=4k+1$ and $L$ is orientable;
\item $n=8k+3, k\neq 2^q$ and $L$ is spin.
\end{enumerate}

\end{theorem}
Let us call an $n$-dimensional, $n>2$, connected closed manifold $L$  with trivial complexified tangent bundle $TL\otimes\C$ {\em admissible} if at least one of the following conditions holds:
\begin{itemize}
\item $n=3$;
\item  $n$ is even, $L$ is orientable and $\chi(L)=2$;  .
\item  $n$ is even, $L$ is not orientable and $\chi(L)$ is even;
 \item $n$ is odd, $L$ satisfies one of the conditions (i)--(iii) of Audin's theorem and $\chi_{\frac12}(L)=1$. 
 \end{itemize}
\begin{theorem}\label{thm:closed-exotic}
Let $L$ be a closed  admissible $n$-dimensional manifold.
Then there exists a Weinstein structure $W(L)= (\om_L,X_L,\phi_L)$ on $T^*S^n$ in the same formal  homotopy  class as the standard one, which contains $L$ as a
flexible Lagrangian submanifold in the homology class of the $0$-section (with $\Z/2$-coefficients in the non-orientable case).
Moreover, infinitely many of the $W(L)$ are distinct as Weinstein manifolds.
\end{theorem}
\begin{remark}\label{rm:necessity}
Conversely, any closed regular Lagrangian in $T^*S^n$ with a possibly exotic, but formally standard Weinstein
structure must have a trivial complexified tangent bundle  and   realizes the generator homology class (with  $\Z/2$-coefficients if $L$ is not orientable), see Lemma \ref{lm:nec-reg}.    Furthermore, if $n$ is even then $\chi(L)=2$ if $L$ is orientable and $\chi(L)$ is even otherwise. Indeed, 
we have $\chi(L)=-[L]\cdot [L]=-[S^n]\cdot [S^n]=\chi(S^n)=2,$ and if $L$ is not orientable this holds $\mod\ 2$.   If $n$ is odd then  one can deduce from Audin's theorem that for all admissible $L$, i.e.  in all cases listed in that theorem   the condition $\chi_{\frac12}(L)=1$ is also necessary.
\end{remark} 
The proof of Theorem \ref{thm:closed-exotic} roughly will proceed as follows: we remove
a disc from $L$ to obtain a manifold $\wh L$ with spherical boundary. Corollary
\ref{cor-3-manifold} produces a flexible Lagrangian embedding $\wh L \hookrightarrow
\C^n$ with parametrized Legendrian boundary $S^{n-1} \cong \p \wh L
\hookrightarrow S^{2n-1}$; if this Legendrian lies in the same formal Legendrian isotopy 
class as the standard unknot, then the result $W(L)$ of attaching a handle to
$\C^n$ along $\p \wh L$  (which contains $L$ as a regular Lagrangian) will be
formally homotopic to $T^* S^n$. It will thus be necessary to understand the
Legendrian isotopy class of the aforementioned Legendrian embedding.

Recall that the formal Legendrian isotopy class of a parameterized Legendrian sphere $g:S^{n-1}\to S^{2n-1}$ in the standard contact $S^{2n-1}$ is determined by two invariants (see   \cite{M11,CE12,Tab88,Go11}): the {\em rotation class} $r(g)\in\pi_{n-1}(U(n))$ and the {\em generalized Thurston-Bennequin invariant} ${\tb}(g)$. If  $n$ is even then $\tb(g)$ can be defined as the linking number between $g$ and its push-off by the Reeb flow. If $n$ is odd the rotation class identically vanishes, while the above    definition of $\tb(S)$ always yields $\pm\frac{\chi(S^{n-1})}2=\pm1$, where the sign depends only on dimension. When $n=3$  there is indeed only 1 formal Legendrian isotopy class of spheres. However, for all odd $n>3$ there are exactly two classes, see \cite{M11, CE12}. They are distinguished by a modified Thurston-Bennequin invariant, which we will continue to denote by $\tb$, and which can be defined as follows, see \cite{CE12}.

The vanishing of $r(g)$ allows us to connect $g$ with the Legendrian unknot $g_0$ by a regular Legendrian homotopy. Viewing the homotopy as an immersed cylinder in $S^{2n-1}\times [0,1]$, and assuming that the immersion has transverse double points, we set  $\tb(g):=k+1(\mod\, 2)$, where $k$ is the number of double points. It turns out that this residue is independent of the choice of a regular homotopy. 


In order to prove Theorem \ref{thm:closed-exotic} we will need the following two Lemmas:
 \begin{lemma}\label{lm:non-orient-change}
 Let $\wh L$ be a non-orientable manifold of dimension $n=2k>2$ bounded by a sphere,  and $h:S^{n-1}\to\p \wh L$ a parameterization of its boundary.
 Suppose that the complexified tangent bundle $TL\otimes\C$ is trivial. Then for any $k\equiv\chi(L) \ (\mod\ 2)$ there exists a Lagrangian embedding
 $f:(\wh L,\p\wh L)\to (B^{2n},
 \p B^{2n})$ with Legendrian boundary such that $\tb(f\circ h)=k$ and $r(\wh f\circ h)=r(f\circ h)$. 
 \end{lemma}
 \begin{proof}

 Let $\wh f:(\wh L,\p\wh L)\to (B^{2n},
 \p B^{2n})$  be a Lagrangian embedding with Legendrian boundary provided by Corollary \ref{cor-3-manifold}. Using the stabilization procedure, see \cite{CE12},
 one can modify for any integer $m$ the Legendrian knot $\wh f|_{\p\wh L}$ by a Legendrian regular homotopy to  a Legendrian embedding $f:\p\wh L\to \p B^{2n}$ with   $\tb(f)=\tb(\wh f|_{\p\wh L})+m$. Note that $r(\wh f)=r(f)$.
  Let   $F:\p\wh L\times[0,1]\to \p B^{2n}\times [0,1]$   be a Lagrangian immersion, corresponding to this regular  homotopy which connects  $F|_{\p\wh L\times 0}=\wh f|_{\p\wh L}$ and
 $F|_{\p\wh L\times 1}=f$. The algebraic number of double points of  $F$   is equal to $m$.  Gluing the Lagrangian cylinder $F$ with the embedding $\wh f$ we get a Lagrangian immersion $\wt f$  of $\wh L$ whose boundary Legendrian sphere has  its Thurston-Bennequin invariant  equal to $\tb(\wh f|_{\p\wh L})+m$ and has the same rotation class as $f$.
 
   Suppose  now that $k$ has the same parity as $\chi(\wh L)$. Since $\tb(\wh f|_{\p\wh L})$ also has the same parity as $\chi(\wh L)$, we have $k=\tb(\wh f|_{\p\wh L})+2l$ for some integer $l$. Apply the above construction to $m:= 2l$. Then, using non-orientability of $\wh L$ one  can  cancel all $2l$ double points in pairs by a smooth (not necessarily Lagrangian) isotopy, thus obtaining a {\em formal} Lagrangian embedding with the required  Legendrian boundary invariants.   Applying again  Corollary \ref{cor-3-manifold}  we  construct  a genuine Lagrangian embedding $f:\wh L\to B^{2n}$ with the prescribed invariants of the boundary.   
  
 \end{proof}

\begin{lemma}\label{lm:formal} 
Suppose that $L$ is admissible, and $\wh L$ is obtained from $L$ by removing an $n$-ball, $\wh L:=L\setminus\Int D^n$. Suppose that the boundary  $\p \wh L$ is parameterized by a diffeomorphism  $h:S^{n-1}\to\p\wh L$. If $L$ is orientable then for any Lagrangian embedding   with Legendrian boundary $f:(\wh L,\p \wh L)\to (B^{2n},S^{2n-1})$, the Legendrian embedding   $f\circ h: S^{n-1} \to S^{2n-1}$  is  in the   {\em formal} Legendrian isotopy class of the   Legendrian unknot $g_0:S^{n-1}\to S^{2n-1}$. If $L$ is non-orientable then there exists a  Lagrangian embedding $f:(\wh L,\p \wh L)\to (B^{2n},S^{2n-1})$  with Legendrian boundary such that $f\circ h$ is in the {\em formal} Legendrian isotopy class of the   Legendrian unknot $g_0$.
  \end{lemma} 

\begin{proof}

 
 Now, suppose that $g=f\circ h$ for a Lagrangian embedding  $f:\wh L\to  B^{2n}$ and a diffeomorphism $h:S^{n-1}\to\p\wh L$. Then $r(g)=0$. 
We already noted that this is always the case when   $n$ is odd.  If $n$ is  even and $L$ is orientable then the Hurewicz 
 homomorphism $\pi_{n-1}(U(n)) \to H_{n-1}(U(n))$     is injective (see e.g.  Theorem 20.9.6 in \cite{Hu94}).  Hence $r(g)=  0$ if the bounding Lagrangian is orientable. But the same argument applies in the non-orientable case to $2r(g)\in \pi_{n-1}(U(n)) =\Z$. 
  
 If $n$ is even  and $L$ is orientable then 
   $\tb(g)=\pm \chi(\wh L)=\tb(g_0)$. This is proven in \cite{Go11} but here is another argument. Consider a vector field $v$ tangent to $\wh L$ such that $v|_{\p\wh L}$ agrees the Liouville vector field $X$.
    Then the push-off of $f(\wh L)$ along $w:=Jdf(v)$ intersects $f(\wh L)$  at $|\chi(\wh L)|$ points. But $w|_S$ is the Reeb vector field, so the linking number entering the definition of  $\tb(g)$ is equal to $\chi(\wh L)$ up to sign. If $\wh L$ is not orientable then the above argument implies only that $\tb(g)\equiv \chi(\wh L)\ (\mod\ 2)$. But in that case Lemma \ref{lm:non-orient-change} allows us to {\em modify} the embedding $f$ to ensure that 
    $\tb(g)= \chi(\wh L)$.  
Suppose now that $n$ is  odd and $n>3$. Then we can  use Audin's Theorem \ref{thm:Audin} to deduce that $\tb(g)$ coincides with the Kervaire semi-characteristic $\chi_{\frac12}(\wh L)$    for all admissible $L$.  
Indeed, let $G:S^{n-1}\times [0,1]\to B^{2n}(R)\setminus B^{2n}(1)$ be a Lagrangian immersion realizing a regular Legendrian homotopy connecting $G|_{S^{n-1}\times0}= g:S^{n-1}\to \p B^{2n}(1)$    with the Legendrian unknot $g_0=G|_{S^{n-1}\times1}:S^{n-1}\to \p B^{2n}(R)$.
 Such immersion exists  for sufficiently large $R$, see \cite{EG-approach}.
 It has $\tb(g)+1\ (\mod\ 2)$ intersection points.
 In turn, the unknot $g_0$ bounds an immersed Lagrangian disc  $g_1:(D^n,\p D^n)\to (\C^n\setminus \Int B^{2n}(R),\p B^{2n}(R))$ with 1 intersection point. Gluing together  the Lagrangian  embedding $f$ with  Lagrangian immersions
 $G$ and $g_1$ we get a Lagrangian immersion of  $L=\wh L\cup D^n$ to $\C^n$
 with  $\tb(g)\ (\mod\ 2)$ intersection points, and hence the claim follows from Audin's Theorem \ref{thm:Audin}.
   \end{proof}
 \begin{proof}[Proof of Theorem \ref{thm:closed-exotic}]
 Using Corollary \ref{cor-3-manifold} we realize $(\wh L:=L\setminus\Int D^n,\p\wh L)$ as a flexible Lagrangian submanifold with Legendrian boundary in the standard symplectic ball $(B^{2n},\p B^{2n}) $. According to Lemma \ref{lm:formal} the gluing diffeomorphism $h$ viewed as a Legendrian embedding $\p D^n\to\p B^{2n}$ is formally Legendrian isotopic to the standard Legendrian unknot.   Hence, by   attaching to the ball
$B^{2n}$ a Weinstein handle of index $n$ along $\p L$ using $h$  we get a  Weinstein domain  $W(L)$ diffeomorphic to the disk cotangent bundle  $UT^*S^n$ with its symplectic structure in  the standard formal symplectic homotopy class.  
The Weinstein domain $W(L )$ contains   $L$ as a closed flexible Lagrangian
submanifold in the homology class of the $0$-section. 

To prove the second part of the theorem, which concerns with infinitely
many symplectomorphism types of the  resulting Weinstein domains $W(L)$, we first
observe that the Viterbo transfer map on symplectic homology  $SH(W(L))\to
SH(UT^*(L))$ is an isomorphism preserving symplectic cohomology's TQFT
operations and BV algebra structure, see e.g. \cite{ritterTQFT} (the part about
compatibility of Viterbo transfer map with the BV operator is folkloric).  In turn,
$SH(UT^*(L))$ as a BV algebra is isomorphic to $H(\Lambda(L); \mathbb{Z})$,
the homology of the free loop space $\Lambda(L)$, with the Chas-Sullivan string
topology BV algebra structure at least whenever $L$ is Spin, see
\cite{viterbo2, salamonweber, abboschwa, abboschwa_prod,
abouzaidsh}.
Since $c_1(T^*S^n) = 0$ and $H^1(T^*S^n; \mathbb{Z}) =0$, there is a canonical grading on $SH(W(L))$; if $H^1(L; \mathbb{Z}) = 0$, then $SH(W(L)) \cong H(\Lambda (L); \mathbb{Z})$ is grading-preserving. 
 In particular, assuming $L, L' $ are spin with $H^1(L; \mathbb{Z}) = H^1(L' ; \mathbb{Z}) = 0$, the Weinstein domains $W(L)$ and $W(L)$ are not symplectomorphic whenever the BV algebras
$H_*(\Lambda(L); \mathbb{Z})$ and $H_*(\Lambda(L' ); \mathbb{Z})$ are non-isomorphic. 
Thus, the above construction provides a rich source of exotic symplectic
structures on $T^*S^n$; it is at least   as rich as the collection of
string topology BV algebra structure on various $n$-manifolds.  

Hence, to get infinitely many  non-symplectomorphic structures it 
suffices to find  for any $n \ge 3$ infinitely many closed stably parallelizable manifolds  $L$ with $H^1(L; \mathbb{Z}) = 0$, $\chi(L) = 2$ for $n$ even, $\chi_{\frac12}(L)=1$ for $n$ odd and $\neq 3$, and different $H_0(\Lambda (L); \mathbb{Z})$. Since the rank of $H_0(\Lambda(L); \mathbb{Z})$ equals the number of conjugacy classes of $\pi_1(L)$, it suffices to find $L$ with fundamental groups $\pi_1(L)$ with different numbers of conjugacy classes; if we furthermore assume that $\pi_1(L)$ is finite, then $H^1(L) = 0$ automatically since $H^1(L)$ is always torsion-free. For $n=3$ we can take the collection of Lens spaces $L(k,1)$. 
To get examples for $n \ge 4$,   consider the CW complex $X$ which has one  0, 1, and 2-cell such that the attaching map for the 2-cell wraps $k$ times around the 1-cell so that $\pi_1(X) \cong \mathbb{Z}/ k \mathbb{Z}$ and  $\chi(X) = 1$.  We can then embed $X$ into $\mathbb{R}^{n+1}$, $n \ge 4$, and take a regular neighborhood  $W$ of $X\subset \mathbb{R}^{n+1}$.  Then the closed  $n$-dimensional manifold $\p W$ satisfies all the required conditions. Indeed, it is stably parallelizable, $\chi(\p W)= 2\chi (W)=2$ for $n$ even,  $\chi_{\frac12}(\p W)=1$ for $n$ odd, and   $\pi_1(\partial W) \cong \pi_1(W) \cong \pi_1(X)=\mathbb{Z}/ k \mathbb{Z}$.
\end{proof}
\begin{remark}\label{rem:mincomplexity}
    There are by now many constructions of exotic Weinstein structures on $T^*
    S^n$, for instance \cite{maydanskiy, mclean, maydanskiyseidel,
    abouzaidseidel}.  A notable feature of the examples $W(L)$ given above is
    that they are each constructed from a {\em single handle attachment} on
    standard $B^{2n}$. Hence they have minimal complexity, meaning the
    defining function for the resulting Weinstein structure can be chosen with
    exactly two critical points (in particular, the Weinstein geometry of each
    $W(L)$ is entirely determined by the Legendrian isotopy class $\partial \wh L
    \subset S^{2n-1}$; recall $\wh L= L \backslash D^{n}$). In addition, each example $W(L)$ contains an
    exact Lagrangian $L$; and hence has non-vanishing symplectic homology with any
    coefficients (at least whenever $L$ is Spin).
 \end{remark}

   It is conceivable that the examples $W(L)$ are as diverse as
    diffeomorphism types of manifolds $L$:  
\begin{problem}\label{prob:completion}
Does symplectic topology of $W(L)$ remember the diffeomorphism type of $L$, or even of
$L$?
\end{problem}
 
As Remark \ref{rem:mincomplexity} recalls, the symplectic topology of
$W(L)$ is entirely determined by the differing Legendrian topology of
embeddings $\partial \wh L \hookrightarrow   \partial B^{2n}$; in particular the examples in
Theorem \ref{thm:closed-exotic} produce infinitely many non-isotopic
Legendrians $\partial \wh L \subset S^{2n-1}$ in the same formal isotopy class.
Hence, one can recast Problem
\ref{prob:completion} in terms of more general questions about the richness of
the Legendrian topology of boundaries $\partial L$ of flexible Lagrangians $L$. For
instance,
 \begin{problem}\label{prob:memory}
Does the Legendrian boundary $\partial L$ of a flexible Lagrangian $L$ remember the topology of the filling?
For instance, when $\p L=S^2$   there exists a unique {\em formal} class of Legendrian 2-spheres in contact $S^5$.
 Does the  genuine Legendrian isotopy class of $\p L$ remember the topology of  $L$?
\end{problem}
Work  in progress of T. Ekholm and Y. Lekili, \cite{EL},  see also  \cite{Laz-new},  implies that the Legendrian
boundary of a flexible Lagrangian knows a lot about the topology of its
filling. For instance, if $L$ is simply connected the  Legendrian isotopy class
of $\p L$ remembers the rational homotopy type of $L$.

Suppose we are given  two $n$-dimensional closed manifolds  manifolds $M,N$  such that
$M$ admits a formal Lagrangian embedding $N\to T M$ which intersects a cotangent fiber at 1 point.  Let  $ \wh M$ and $\wh N$ be manifolds with boundary obtained by removing   small $n$-discs from $M$ and $N$.  Then   the cotangent bundle $W:=T^*\wh M$  is a subcritical Weinstein manifold, and hence Theorem \ref{thm:main} provides a flexible Lagrangian embedding $f:(\wh N,\p \wh N)\to (W,\p W)$
 with Legendrian boundary.
 We also have the canonical inclusion
$(\wh M,\p\wh M)\hookrightarrow (W,\p W)$.
This leads to an alternative:

{\em either  Legendrian spheres $f(\p\wh N)$ and $j(\p\wh M)$ are not Legendrian isotopic, or the nearby Lagrangian conjecture fails}.

In particular, an intriguing case  is when $M,N$ are two homeomorphic  $4$-manifolds distinguished by gauge-theoretic invariants.

 

\section{Murphy-Siegel example}\label{sec:murphysiegel}

\begin{figure}[h]
\begin{center}
\includegraphics[width=50mm]{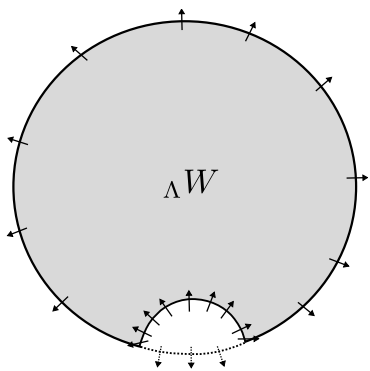}
\caption{Caving out a neighborhood of a Legendrian submanifold}\label{fig:caving-out}
\end{center}  
\end{figure}

\begin{figure}[h]
\begin{center}
\includegraphics[width=50mm]{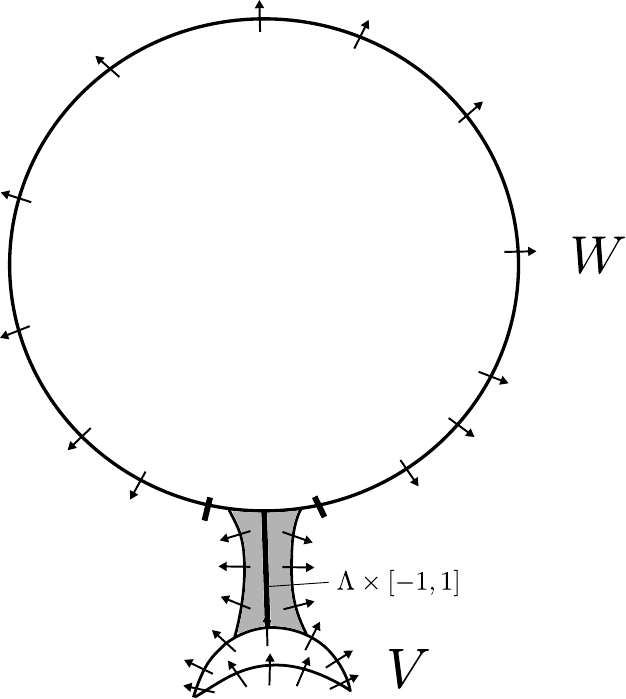}
\caption{Attaching a cylindrical handle}\label{fig:caving_surgery}
\end{center}  
\end{figure}

We note that  when $(D,\p D)\subset (W,\p W)$ is a regular Lagrangian disc, then besides the complementary
{\em cobordism } $W_D$ we can also consider a Weinstein {\em domain} $W^D$, 
the result of Weinstein {\em anti-surgery}. In other words, $W$  is obtained from $W^D$ by a Weinstein surgery with the co-core disc $D$. We denote by $\Gamma$ the Legendrian sphere in $W^D$ along which the handle with co-core $D$ is attached.

\begin{lemma}\label{lm:two-complements}
If $W_D$ is a flexible cobordism, then $W^D$ is a flexible domain. If $W^D$ is a flexible domain and $W$ is obtained from $W^D$ by attaching a handle along a loose Legendrian sphere $\p D\subset\p W^D$, then $W_D$ is flexible.
\end{lemma}
Lemma \ref{lm:two-complements} follows from the more general   Lemma
\ref{lm:reg-Leg} below.  (Note alternately that the second part of Lemma
\ref{lm:two-complements} is equivalent to Proposition \ref{prop:flex}(i)).

Suppose $W$ is a flexible Weinstein domain and $\Lambda\subset \p W$ is a Legendrian submanifold. Then one can canonically construct a Liouville cobordism 
${}_\Lambda W$ by {\em caving out} a neighborhood of $\Lambda$, see Fig.~\ref{fig:caving-out}.
Thus $\p_-({}_\Lambda W)$  is  the canonical Darboux neighborhood $J^1(\Lambda)$ of $\Lambda$ in $\p W$.  Note that in the situation  of Lemma \ref{lm:two-complements} the Weinstein cobordism $W_D$ coincides with the cobordism ${}_{\Gamma}W^D$.
\begin{lemma}\label{lm:reg-Leg}
  ${}_\Lambda W$ is always a  Weinstein cobordism.  If ${}_\Lambda W$ is flexible then $W$ is flexible. If $W$ is flexible and  $\Lambda$ is loose then  ${}_\Lambda W$ is flexible.
 \end{lemma}

  \begin{proof} 
    Take the cotangent bundle $T^*(\Lambda\times[-1,1])$ with its canonical subcritical Weinstein domain structure $U$ of a cotangent bundle of a manifold with boundary. The boundary $\p U$ contains $\Lambda\times 1$ as its Legendrian submanifold.
  Consider a tubular neighborhood $\Sigma$ of  $\Lambda\times 1$ in $\p U$.  It can be identified with a neighborhood  of $\Lambda$ in $J^1(\Lambda)$. Consider a trivial Weinstein cobordism  $V$ over $\Sigma$, or rather its sutured version, see Fig.~\ref{fig:caving-out}. 
  Note that the positive boundary $\p_+V$ is a copy of $\Sigma$ and it  coincides with the negative boundary $\p_-({}_\Lambda W)$.  The Weinstein handlebody presentation of   ${}_\Lambda W$ builds   ${}_\Lambda W$ by a sequence of handle attachments to $\p_+V$. Making the same handle attachments to $\Sigma\subset\p U$ we build instead  the original   Weinstein domain $W$. If  ${}_\Lambda W$ is flexible then the  result of the gluing is flexible  as well, because we added a flexible cobordism to a subcritical domain.
  
  Let us now assume that  $W$ is flexible and $\Lambda$ is loose. Then we  can get 
  ${}_\Lambda W$ by attaching  a  handle with the  cylindrical Lagrangian core $\Lambda\times[-1,1]$ to the cobordism
  $V\sqcup W$ by gluing $\Lambda\times (-1)$ to the $0$ section $\Lambda\subset \Sigma\times 1$ and $\Lambda\times (+1)$ to $\Lambda\subset\p W$, see 
  Fig.~\ref{fig:caving_surgery}.
     By choosing a handlebody decomposition of $\Lambda\times[-1,1]$ with exactly one $n$-handle, we can decompose the attachment of the handle with a cylindrical core to a sequence of attachments of handles corresponding  to the handles of the decomposition of $\Lambda\times[-1,1]$. One can arrange that the stable manifolds of the subcritical handles are in the complement of the intersection of $\Lambda\subset W$ with its loose chart. It then follows that the only index $n$ handle is attached along a loose knot, and hence the resulting cobordism is flexible.
    For general $\Lambda$, this construction shows that ${}_\Lambda W$ is always a Weinstein cobordism.        
         \end{proof}
\begin{problem}\label{prob:reg-Leg-flex}
Suppose ${}_\Lambda W$ (and hence $W$)
is   flexible. Is it true that then
$\Lambda$ is   a loose Legendrian?
 \end{problem}

Note that an affirmative answer to this problem would also give an affirmative answer to Problem \ref{prob:converse-to-flex}(i).

Emmy Murphy and Kyler Siegel produced an example, see \cite{MS15}, of a non-flexible Weinstein domain $W$  which becomes flexible after attaching a Weinstein  $n$-handle $H$ (they call  domains with this property {\em subflexible}. Denote
$\wh W:=W\cup H$. Let $(C,\p C)\subset (\wh W,\p \wh W)$ be the Lagrangian co-core of the handle $H$.

\begin{prop}\label{prop:MS-example}
 $(C,\p C)\subset (\wh W,\p \wh W)$ is non-flexible.
 \end{prop}
 \begin{proof} We can identify $W$ with $\wh W^{C}$. But by Lemma \ref{lm:two-complements} the non-flexibility of $ \wh W^{C}$ implies the non-flexibility of $\wh W_C$.
 \end{proof}
 
 Thus, {\em in a flexible domain there could be non-flexible regular Lagrangian discs.}

Moreover, Murphy-Siegel's results imply that even  in the standard symplectic ball there exists a non-flexible but regular union of Lagrangian discs.

Let  $(\wh C,\p \wh C)\subset (\wh W,\p \wh W)$ be the flexible realization of the same formal Lagrangian class using Theorem \ref{thm:main}. It follows that $(C, \p C)$ and $(\wh C,\p \wh C)$ are not Hamiltonian isotopic or even symplectomorphic.
\begin{problem}\label{prob:MS-example-boundary}
 Are the spheres $\p C,\p\wh C\subset\p\wh W$ Legendrian isotopic?
\end{problem}
If the answer was positive, then by attaching an $n$-handle we would get a Weinstein manifold containing two regular Lagrangian spheres in the same formal class, one with a flexible and the other with a non-flexible  complement, thus providing a negative answer to the last part of Problem \ref{prob:around-nearby}.
\begin{remark}\label{rm:relativeMS}
One can ask if there is a relative version of the Murphy--Siegel phenomenon  \cite{MS15}, i.e., are there non-flexible regular Lagrangians  which become flexible after attaching a handle to the Weinstein domain?
 One example of this kind  is  immediate from \cite{MS15}. Namely, let $W$ be a subflexible (but not flexible) Weinstein domain, which becomes flexible after attaching a handle $H$. As work of Murphy-Siegel shows, one can attach to $W$ an index $n$ handle $\wt H$ along a loose knot in such a way that the new domain $\wt W:=W\cup \wt H$ remains non-flexible. Note that the co-core Lagrangian disc $C\subset \wt H$ is regular but not flexible in $\wt W$. However, it is flexible in $\wt W\cup H$ (by the same direct argument that shows that $W \cup H$ is flexible). This construction is not a completely satisfactory answer to the above problem; note that while the co-core disc $C$ is  not flexible in $\wt W$, it is {\em semi-flexible} there, see Section \ref{sec:semiflexible} below, and thus already has vanishing Floer-theoretic invariants. \end{remark}
\section{Semi-flexible Lagrangians}\label{sec:semiflexible}
  Suppose $L\subset W$ is a flexible Lagrangian cobordism in a Weinsein cobordism $W$, and a  $W'$ any  other Weinstein cobordism.   Suppose that $\wt W$ obtained from $W$ and $W'$ by gluing some connected components of $\p_-W$ not intersecting $\p_-L$ with the corresponding components of $\p_+W'$. We then say that $L\subset \wt W$ is a {\em semi-flexible} Lagrangian cobordism.

 A prototypical example of a semi-flexible Lagrangian is  the {\em  co-core} of an index $n$ Weinstein handle  attached along a loose Legendrian sphere $S\subset\p_+W$ to a Weinstein cobordism $W$.

 \begin{problem}\label{prob:semi-in-flex}
 Are there semi-flexible but not flexible Lagrangians in a flexible Weinstein cobordism?
  \end{problem}
   In particular, is it possible to make a non-flexible Weinstein cobordism flexible by  attaching  handle along a loose knot?

 \begin{prop}\label{prop:int-closed}
 If $L$ is semi-flexible Lagrangian with $\p L\neq\varnothing$ in a Weinstein domain $W$, then there are no closed  Lagrangians  $L'$ with non-zero homological  intersection  index
 $L\cap L'$.
 \end{prop}
 \begin{lemma}\label{vanishingwrapped}
Let $L$ be  a semi-flexible Lagragian  in a Weinstein domain $W$. Suppose that $\p L\neq\varnothing$.
Then  the wrapped Floer homology $WFH(L,L;W)$ vanishes.
 \label{lm:WFH}
 \end{lemma}
 \begin{proof}
For the canonical subcritical neighborhood $U\supset L$  we have $WFH(L,L;U)=0$ because it is a (unital) module over symplectic homology $SH(U)=0$.
 By assumption $W$ is obtained by attaching subcritical or flexible handles to the disjoint
 union $U\sqcup W'$ along spheres in the complement of $\p L$. But $WFH(L,L;U\sqcup W')= WFH(L,L;U)=0$, and attaching of subcritical or flexible handles does not change $WFH(L,L)$. For index $n$ handles this last fact is proven  in  \cite{BEE12}.  In the subcritical case this is a part of symplectic folklore (see \cite{Cie02, MLYau,  EkNg} for the closest statements in the literature). Hence, $WFH(L,L;W)=0$.
 \end{proof}
 \begin{proof}[Proof of Proposition \ref{prop:int-closed}]
 For any pair $L, L'$ with $L'$ closed, the Euler characteristic  of the Lagrangian Floer
 homology $FH(L,L')$ coincides with the homological intersection index $L\cap
 L'$. On the other hand,  $FH(L,L') = WFH(L,L')$ is a (unital) module over
 $WFH(L,L)$ which vanishes by Lemma \ref{lm:WFH}, hence $FH(L,L')$ must vanish
 as well.
 \end{proof}
 
 \begin{cor}\label{cor:no-semi}
Semi-flexible Lagrangians do not satisfy the existence $h$-principle. For instance, for any closed $L$ the class of a cotangent fiber in $T^*L$ cannot be realized by a semi-flexible Lagrangian disc with Legendrian boundary.
 \end{cor}
 
It seems similarly unlikely that semi-flexible Lagrangians satisfy any non-trivial  form of the 
uniqueness  statement.
However, we formulate this question as an open problem.

 \begin{problem}
Consider two semi-flexible Lagrangians $L_0,L_1\subset (W, \om)$.  Suppose that there exists a diffeomorphism $W\to W$ such that $f(L_0)=L_1$, and $f^*\om$  and  $\om$  are homotopic  as non-degenerate not necessarily closed $2$-forms. Is  $f$ isotopic to a symplectomorphism $g:W\to W$ with $g(L_0)=L_1$?
 
 \end{problem}
  We finish this section with Proposition \ref{prop:semi-flex-discs}  which gives a  Hamiltonian isotopy  classification  of semi-flexible Lagrangian discs   smoothly isotopic to a disc  in $\p_+W$. Consider $\R^{2n}$ with the standard Liouville form $\lambda=\frac12\left(\sum\limits_1^nx_idy_i-y_idx_i\right)$. In the unit sphere $S:= \{\sum\limits_1^n x_i^2+y_1^2=1\}$ consider a Legendrian equator $E:= S\cap\{y=(y_1,\dots y_n)=0\}$ 
      and the Lagrangian disc $\wt{E}$ bounded by $S$ in the unit ball $B:=\{\sum\limits_1^n x_i^2+y_1^2\leq 1\}\subset\R^{2n}$. 
  Denote by $C$ the pre-Lagrangian disc $ S\cap\{y_1\geq0, y_2=\dots=y_n=0\}$ which bounds $E$ in $S$. Note that 
  $$\lambda|_C= \frac12(x_1dy_1-y_1dx_1)|_C =-
  y_1^2d\left(\frac{x_1}{y_1}\right)|_C =\left(-1+\sum\limits_1^n x_j^2\right)d\left(\frac {x_1}{\sqrt{1-\sum\limits_1^n x_i^2}}\right).$$
 The disc $C$ is foliated by Legendrian discs $x_1=c\, y_1$, $c\in\R$. The Lagrangian lift of $C$ to $\R^{2n}\setminus 0$ (thought of as the symplectization of $(S,\{\lambda|_S=0\})$) is the plane $\{y_1=1, y_2=\dots= y_n=0\}$.
  We say that a Legendrian sphere $S$ in a contact manifold $(Y,\xi=\{\beta=0\})$ is a {\em Legendrian unknot} if it bounds in $Y$ a disc $D$ such that $\beta|_D$ is isomorphic to a form proportional to $\lambda|_C$. This definition is equivalent to the usual definition of the unknot  
 as a Legendrian sphere which is Legendrian isotopic to a sphere in a Darboux chart with the saucer-like front projection.  Any unknot bounds in the symplectization of  $Y$ (and hence in any Liouville filling of $Y$) a Lagrangian  disc $\wt D$ which is the lift of the pre-Lagrangian disc $D$. We call any Lagrangian disc Hamiltonian isotopic to $\wt D$ {\em small}.
 We recall that by Hamiltonian isotopy we mean the Hamiltonian isotopy of completed Lagrangians, which is equivalent
 to a Lagrangian homotopy of discs with Legendrian boundary.
 
     {\begin{prop}\label{prop:semi-flex-discs}
  Any   semi-flexible Lagrangian disc $(D,\p D) \subset (W,\p_+W)$ with Legendrian boundary  which is smoothly  isotopic to  a disc in $\p_+W$ is small.
    \end{prop}
    \begin{proof}
   By assumption, $W$ can be built by attaching a   flexible cobordism $V$ to   the disjoint union of
   $T^*D\sqcup W'$, where $W'$ is a Weinstein domain. Note that $D$ is small in $T^*D$ and thus the Legendrian sphere $\p D$ bounds a pre-Lagrangian disc $E \subset \p T^*D$.  The attaching spheres of all subcritical handles
   forming the flexible cobordism $V$ are generically disjoint from $E$ by dimension reasons.
   On the other hand, the attaching spheres of index $n$ handle can be disjoined from $E$ by a smooth isotopy in view of the condition that $D$ can be pushed by a smooth isotopy to the boundary $\p_+W$. But then the  flexibility condition, meaning here the looseness of the Legendrian link formed by   the attaching spheres of all index $n$ handles, allows us to disjoin  the attaching spheres of all   index $n$ handles from $E$  via a Legendrian isotopy. Hence, $\p D$ bounds the required pre-Lagrangian disc $E$ in $\p_+W$, and thus $D$ is small.
     \end{proof}
\medskip  
\subsubsection*{Acknowledgements}
Part of this paper was written when two of the authors visited Mittag-Leffler Institute in Djursholm, Sweden. They thank the Institute for hospitality and the participants of the program on symplectic and contact topology, especially Tobias Ekholm and Yank{\i} Lekili, for stimulating discussions.
The authors are grateful to Emmy Murphy and Kyler Siegel for sharing the 
 results of their paper \cite{MS15} before it was posted. 



\begin{thebibliography}{99}
    \bibitem{abouzaidsh} M. Abouzaid, {\em Symplectic cohomology and Viterbo's theorem}, arXiv:1312.3354.
    
    \bibitem{abouzaidseidel} M. Abouzaid and P. Seidel, {\em Altering symplectic manifolds by homologous recombination}, arXiv:1007.3281, 2010.
    

\bibitem{Ab03} H. Abbaspour, {\em On String Topology of Three Manifolds}, Topology, {\bf 44}(2005),  1059--1091.

\bibitem{abboschwa} A. Abbondandolo and M. Schwarz, {\em On the Floer homology of cotangent bundles}, Comm. Pure Appl. Math. {\bf 59} (2006), 254--316.

\bibitem{abboschwa_prod} A. Abbondandolo and M. Schwarz, {\em Floer homology of cotangent bundles and the loop product}, Geom. Topol. {\bf 14} (2010), 1569-1722.
\bibitem{Audin} M.~Audin, {\em Fibr\'es  normaux d'immersions en dimension double, points doubles d'immersions lagrangiennes et plongements totallment r\'eels},
Comm. Math. Helvet. {\bf 63}(1988), 593--623.
    \bibitem{AMP} D. Auroux, V. Mu\~{n}oz and F. Presas, {\em Lagrangian submanifolds and Lefschetz pencils}, J. Symplectic Geom. {\bf 3} no. 2, 171--219 (2005)

\bibitem{BEE12} F. Bourgeois, T. Ekholm and  Y. Eliashberg,
{\em Effect of Legendrian Surgery}, 
Geom.  Topol., {\bf 16}(2012), 301--389.
\bibitem{Cie02} K. Cieliebak, {\em Handle attaching in symplectic homology and the Chord Conjecture}, J.
Eur. Math. Soc.. {\bf 4}(2002), 115--142.
\bibitem{CE12} K.~Cieliebak and Y.~Eliashberg, {\em From Stein to
    Weinstein and Back -- Symplectic Geometry of Affine Complex
    Manifolds}, Colloquium Publications Vol.~59,
  Amer. Math. Soc. (2012). 

  \bibitem{E12}T.~Ekholm, {\em Rational SFT, linearized Legendrian contact homology, and Lagrangian Floer cohomology}, Perspectives in analysis, geometry, and topology, 109--145, Progr. Math., 296, Birkh\"{a}user/Springer, New York, 2012.

  \bibitem{EL} T. Ekholm and Y. Lekili, in preparation.

 \bibitem{EkNg} T. Ekholm and L. Ng, {\em Legendrian contact homology in the boundary of a subcritical Weinstein 4-manifold},  J. Diff. Geom. {\bf 101}(2015), 67--157.

  \bibitem{Eli90} Y.~Eliashberg, {\em Topological characterization of
Stein manifolds of dimension $>2$}, Internat.~J.~Math.~{\bf 1},
no.~1, 29-46 (1990). 
  \bibitem{EliGro91} Y.~Eliashberg, M~Gromov, {\em Convex symplectic manifolds},    Proc.  Symp.  Pure
Math., {\bf  52}(1991), Amer. Math. Soc., Providence, RI,  135--162.
  \bibitem{EG-approach} Y.~Eliashberg, M~Gromov, {\em Lagrangian intersection theory. Finite-dimensional approach},    AMS Transl.,  {\bf 186}(1998), N2, 27--116.
\bibitem{EM-caps} Y.~Eliashberg and E. Murphy, Lagrangian caps,
 {\it Geom. and Funct. Anal.}, 2013, (DOI) 10.1007/s00039-013-0243-6.
 \bibitem{EM-symp}  Y.~Eliashberg and E. Murphy, {\em Making cobordisms symplectic},  	arXiv:1504.06312.
 \bibitem{GP15} E. Giroux and J. Pardon, {\em Existence of Lefschetz fibrations on Stein and Weinstein domains}, arXiv:1411.6176. 
 \bibitem{Go11} R. Golovko, {\em A note on Lagrangian cobordisms between Legendrian submanifolds in $\R^{2n+1}$},  Pacific  J. Math., {\bf 261}(2013),  101--116.  
 \bibitem{Hu94} D. Husem\"{o}ller, {\em Fiber bundles}, Springer-Verlag, New York, 1994.
 \bibitem{Laz-new} O. Lazarev, {\em Contact manifolds with flexible fillings}, preprint 2016.
 \bibitem{maydanskiy} M. Maydanskiy, {\em Exotic symplectic manifolds from Lefschetz fibrations}, Ph.D. thesis, Massachusetts Institute of Technology, 2009.
 \bibitem{maydanskiyseidel} M. Maydanskiy and P. Seidel, {\em Lefschetz fibrations and exotic symplectic structures on cotangent bundles of spheres}, J. Topology, {\bf 3} (2010), 157--180.
  \bibitem{mclean} M. McLean, {\em Lefschetz fibrations and symplectic homology}, Geom. Topol. 13 (2009), No. 4, 1877–1944.
   \bibitem{M11} E.~Murphy, {\em Loose Legendrian embeddings in high
    dimensional contact manifolds}, arXiv:1201.2245. 
    \bibitem{M13}  E.~Murphy, {\em Closed exact Lagrangians in the symplectization of contact manifolds}, arXiv:1304.6620.
\bibitem{MS15} E.~Murphy and K.~Siegel, {\em Subflexible symplectic manifolds}, arXiv:1510.01867.

    \bibitem{ritterTQFT} A. Ritter, {\em Topological quantum field theory structure on symplectic cohomology}, J. Topol. {\bf 6} (2013), 381--489.

   \bibitem{Ri12} G. Dimitroglou Rizell, {\em Legendrian ambient surgery and Legendrian contact homology},
   arXiv:1205.5544v5.

   \bibitem{salamonweber} D. Salamon and J. Weber, {\em Floer homology and the heat flow}, Geom. Funct. Anal., {\bf 16} (2006), 1050-1138.
   \bibitem{Smale} S. Smale, {\em On the structure of manifolds}, Amer. J. Math., {\bf 84}(1962) pp. 387--399.
   \bibitem{Tab88} S. Tabachnikov, {\em An invariant of a submanifold that is transversal to a distribution} (Russian), Uspekhi Mat. Nauk 43 {\bf 3} (1988), 193–-194; translation in Russian Math. Surveys 43, {\bf 3} (1988), 225–-226.
   \bibitem{viterbo2} C. Viterbo, {\em Functors and computations in Floer homology with applications. II}, preprint.
       \bibitem{viterbo1} C. Viterbo, {\em Functors and computations in Floer homology with applications. I}, Geom. Funct. Anal. {\bf 9} (1999), 985--1033.
\bibitem{Wei91}A.~Weinstein, {\em Contact surgery and symplectic
handlebodies}, Hokkaido Math.~J.~{\bf 20}(1991), 241--251.
\bibitem{MLYau} M.-L. Yau, {\em Cylindrical contact homology of subcritical Stein-fillable contact manifolds}, Geom. Topol. {\bf 8}(2004), 1243--1280.
\end{thebibliography}
\end{document}